\documentclass[11pt]{article}
\usepackage{amsmath, amssymb, amsfonts, amsthm,latexsym}
\usepackage{graphicx}

\newtheorem{theorem}{Theorem}[section]

\newtheorem{lemma}[theorem]{Lemma}
\newtheorem{propos}[theorem]{Proposition}
\newtheorem{cor}[theorem]{Corollary}

\DeclareMathOperator*{\essinf}{ess\,inf}
\newcommand{\overbar}[1]{\mkern 1.5mu\overline{\mkern-1.5mu#1\mkern-1.5mu}\mkern 1.5mu}

\DeclareMathOperator{\trig}{trig}

\title{Chebyshev-type Quadratures for Doubling Weights}
\author{Shoni Gilboa\thanks{Mathematics Department, The Open University of Israel, Raanana 43107, Israel.}
\and Ron Peled\thanks{School of Mathematical Sciences, Tel Aviv
University, Tel-Aviv 69978, Israel. Supported by an ISF grant and an
IRG grant.}}

\begin{document}
\maketitle
\begin{abstract}
A Chebyshev-type quadrature for a given weight function is a
quadrature formula with equal weights. In this work we show that a
method presented by Kane may be used to determine the order of
magnitude of the minimal number of nodes required in Chebyshev-type
quadratures for doubling weight functions. This extends a long line
of research on Chebyshev-type quadratures starting with the 1937
work of Bernstein.
\end{abstract}

\renewcommand{\thefootnote}{\fnsymbol{footnote}}
\footnotetext{\emph{Keywords:} Chebyshev-type quadrature formula,
doubling weights, equal weight quadrature.} \footnotetext{\emph{2010
Mathematics Subject Classification:} 65D32, 41A55.}
\renewcommand{\thefootnote}{\arabic{footnote}}

\section{Introduction}

Let $w$ be a non-negative, integrable function on the interval
$[-1,1]$, such that $\int_{-1}^1 w(t)dt>0$. Such a $w$ will be
called a \emph{weight function}. A \emph{quadrature} (formula) of
degree $n$ is a sequence of points $-1\leq t_1\le t_2\le \ldots\le
t_N\leq 1$, called nodes, and a sequence of weights
$w_1,w_2,\ldots,w_N$ such that
\begin{equation}\label{eq:quadrature_def}
\int_{-1}^1p(t)w(t)dt=\sum_{i=1}^Nw_i p(t_i)
\end{equation}
for every polynomial $p$ of degree at most $n$. If
$w_1=w_2=\cdots=w_N$ then the quadrature is called a
\emph{Chebyshev-type quadrature}, or an equal-weight quadrature.

Given $w$ and the location of the nodes, the equalities
\eqref{eq:quadrature_def} reduce to a set of $n$ linear equations in
$N$ unknowns and it is simple to see from this that a quadrature
formula exists whenever $N\ge n$. The situation becomes more
complicated if we require the quadrature formula to have
non-negative weights. However, the celebrated Gaussian quadrature
formula (see, e.g., \cite[Chapter IV, Section 8]{Karlin-Studden})
satisfies this restriction and achieves the optimal degree $n =
2N-1$.

Surprisingly, the situation changes dramatically if we require the
quadrature to be of Chebyshev-type, i.e., to have equal weights.
This was established by Bernstein \cite{Bernstein1, Bernstein2} in
1937 who proved that for a constant weight function the minimal
possible number of nodes in a Chebyshev-type quadrature of degree
$n$ is of order $n^2$.

Bernstein's result naturally raises the question of understanding
the minimal number of required nodes in Chebyshev-type quadratures
for other weight functions. For a weight function $w$ and a positive
integer $n$, we denote by $N_w(n)$ the minimal number of nodes in a
Chebyshev-type quadrature of degree $n$ for $w$. The most
well-studied case is that of the \emph{Jacobi weight function},
\begin{equation}\label{eq:Jacobi_weight}
w_{\alpha,\beta}(x) := (1-x)^\alpha(1+x)^\beta,\quad
\alpha,\beta>-1.
\end{equation}
Kuijlaars \cite{Kuijlaars2} (see also \cite{Kuijlaars1}) generalized
Bernstein's result by proving that $N_{w_{\alpha,\beta}}(n)$ is of
order $n^{2 + 2\max\{\alpha, \beta\}}$ when $\alpha,\beta\ge 0$.
Kane \cite{Kane} recently extended this result to $\alpha, \beta\ge
-1/2$ (the case $\alpha=\beta=-1/2$ is classical as the Gaussian
quadrature itself already has equal weights). In the regime
$-1<\alpha,\beta\le-\frac{1}{2}$, Kuijlaars \cite{Kuijlaars4} proved
that there exists a $\lambda_0>0$ such that if
$\alpha=\beta=-\frac{1}{2}-\lambda$ or $\alpha = -\frac{1}{2}$ and
$\beta=-\frac{1}{2}-\lambda$ for some $\lambda<\lambda_0$ then
$N_{w_{\alpha,\beta}}(n)$ has order $n$ (and, moreover,
$N_{w_{\alpha,\beta}}(n)\le n+2$).

Regarding results on general classes of weight functions, Kuijlaars
proved a universality result in \cite{Kuijlaars analytic}, showing
that for every weight function $w$ of the form $h(x)(1-x^2)^{-1/2}$,
where $h$ is positive on $[-1,1]$ and analytic in a neighbourhood of
$[-1,1]$, $N_w(n)$ is of order $n$. Geronimus \cite{Geronimus},
continuing ideas of Bernstein and Akhiezer, gave \emph{lower} bounds
for $N_w(n)$ for certain general classes of weight functions. Wagner
\cite{Wagner} gave \emph{upper} bounds on $N_w(n)$ for weight
functions which are bounded below by a constant multiple of the
weight $w_{\alpha, \alpha}$, $\alpha\ge 0$, and bounded above by a
constant. The results in \cite{Wagner} are rather general in that no
further assumptions are placed on the weight function (and,
additionally, more general types of quadrature formulas are
considered). They are, however, somewhat difficult to apply and may
overestimate the order of magnitude of $N_w(n)$ (as $n$ tends to
infinity). An upper bound with similar merits and disadvantages was
given by Rabau and Bajnok \cite{Rabau-Bajnok} (inspired by
\cite{Reyna}). In addition, the work \cite{Peled} gives a simple
\emph{upper} bound on $N_w(n)$ for \emph{every} weight function (and
more generally, every measure). The bound of \cite{Peled} is,
however, far from sharp in many interesting cases, being, for
instance, exponential in $n$ when $w$ is the constant weight
function.

In this work we determine the order of magnitude of $N_w(n)$ (as $n$
tends to infinity) in great generality, extending most of the
previous results on this question. Our results are obtained by
combining topological tools developed by Kane \cite{Kane} with
analytic properties of doubling weight functions given by
Mastroianni and Totik \cite{MastroianniTotik1}.

A weight function $w$ on $[-1,1]$ is called \emph{doubling} if there
is an $L>0$ (necessarily $L\geq 2$), called the \emph{doubling constant}, such that for
every $\delta>0$ and $-1\leq a\leq 1$,
$$\int_{a-2\delta}^{a+2\delta}w(t)dt\leq L\int_{a-\delta}^{a+\delta}w(t)dt,$$
where $w(t)$ is interpreted as $0$ for $t\notin[-1,1]$.

For a weight function $w$ and a positive integer $n$, in addition to
the definition of $N_w(n)$ given above, we denote by
$\overbar{N}_w(n)$ the minimal number such that for every $N\ge
\overbar{N}_w(n)$ there is a Chebyshev-type quadrature of degree $n$
for $w$ with exactly $N$ nodes. We trivially have
$N_w(n)\le\overbar{N}_w(n)$ but it may happen that the two
quantities have radically different orders of magnitudes; see
F\"orster \cite{Forster} for an example where $N_w(n)$ is linear in
$n$ while $\overbar{N}_w(n)$ is exponential in $n$.

\begin{theorem}\label{thm:general} Suppose $w$ is a doubling weight function on $[-1,1]$ with doubling constant $L$.
 Let
$$R_w(n):=\frac{\int_{-1}^1 w(t)dt}{\inf_{-1\leq x\leq 1}\int_{x-\Delta_n(x)}^{x+\Delta_n(x)}w(t)dt},$$
where
$\Delta_n(x):=\frac{1}{n}\left(\sqrt{1-x^2}+\frac{1}{n}\right)$ and
$w(t)$ is again interpreted as $0$ for $t\notin[-1,1]$. Then
\begin{equation*}
c(L)R_w(n)\leq N_w(n) \leq \overbar{N}_w(n)\leq C(L)R_w(n),\quad
n\ge 1,
 \end{equation*}
where $c(L)$ and $C(L)$ are positive constants depending only on
$L$.
\end{theorem}

Many of the weight functions that appear in analysis are doubling.
These include the \emph{generalized Jacobi weight functions} which
are the weight functions having the form
\begin{equation}\label{eq:generalized_Jacobi_weights}
h(x)(1-x)^{\alpha}(1+x)^{\beta}|x-s_1|^{\gamma_1}\ldots|x-s_{\ell}|^{\gamma_{\ell}},
\end{equation}
with $h$ a positive measurable function bounded away from zero and
infinity, $\ell\ge 0$, $-1<s_1<\ldots<s_{\ell}<1$ and
$\alpha,\beta,\gamma_1,\cdots,\gamma_\ell>-1$. The Jacobi weight
function $w_{\alpha,\beta}$ given by \eqref{eq:Jacobi_weight} is
obtained by setting $h\equiv 1$ and $\ell=0$.
Theorem~\ref{thm:general} shows that weight functions $w$ of this
form satisfy that
\begin{equation*}
  \text{both $N_w(n)$ and $\overbar{N}_w(n)$
are of order
}n^{\max\{1,2(\alpha+1),2(\beta+1),\gamma_1+1,\ldots,\gamma_{\ell}+1\}}.
\end{equation*}
This extends the known results even in the Jacobi weight case, where
the order of magnitude of $N_{w_{\alpha,\beta}}(n)$ was not
previously known for all values of $-1<\alpha,\beta\le
-\frac{1}{2}$. Moreover, it shows that the special form of the
Jacobi weight function plays no role in this problem beyond the type
of its zeros and singularities.

Theorem~\ref{thm:general} is a direct consequence of an analogous
theorem for weight functions on the unit circle. This result, which
we shall now describe, seems more natural in that the corresponding
estimate takes a simpler form. A weight function on the unit circle,
or \emph{$2\pi$-periodic weight function}, is a non-negative,
$2\pi$-periodic measurable function $W$ on the real line, such that
$0<\int_{-\pi}^{\pi} W(\theta)d\theta<\infty$. A \emph{trigonometric
quadrature} (formula) of degree $n$ for such a $W$ is a sequence of
real numbers $\theta_1\le \theta_2\le \ldots\le \theta_N$, called
nodes, and a sequence of weights $w_1,w_2,\ldots,w_N$ such that
\begin{equation}\label{eq:_trig_quadrature_def}
\int_{-\pi}^{\pi}p(\theta)W(\theta)d\theta=\sum_{i=1}^N w_i
p(\theta_i)
\end{equation}
for every trigonometric polynomial $p$ of degree at most $n$. If
$w_1=w_2=\cdots=w_N$ then the quadrature is called a
\emph{Chebyshev-type trigonometric quadrature}. A $2\pi$-periodic
weight function $W$ is called \emph{doubling} if there is an $L>0$ (necessarily $L\geq 2$),
called the \emph{doubling constant}, such that for every $\delta>0$
and real $a$,
\begin{equation*}
\int_{a-2\delta}^{a+2\delta}W(\theta)d\theta\leq
L\int_{a-\delta}^{a+\delta}W(\theta)d\theta.
\end{equation*}

For a $2\pi$-periodic weight function $W$ and a positive integer
$n$, we denote by $N^{\trig}_W(n)$ the minimal number of nodes in a
Chebyshev-type trigonometric quadrature of degree $n$ for $W$. We
also denote by $\overbar{N}^{\trig}_W(n)$ the minimal number such
that for every $N\ge \overbar{N}^{\trig}_W(n)$ there is a
Chebyshev-type trigonometric quadrature of degree $n$ for $W$ with
exactly $N$ nodes.

\begin{theorem}\label{thm:trig} Suppose $W$ is a doubling $2\pi$-periodic weight function with doubling constant $L$. Let
\begin{equation}\label{eq:R_trigonometric}
R^{\trig}_W(n):=\frac{\int_{-\pi}^{\pi}
W(\theta)d\theta}{\inf_{x\in\mathbb{R}}\int_{x-\frac{1}{n}}^{x+\frac{1}{n}}W(\theta)d\theta}.
\end{equation}
Then
\begin{equation*}
c(L)R^{\trig}_W(n)\leq
N^{\trig}_W(n)\leq\overbar{N}^{\trig}_W(n)\leq
C(L)R^{\trig}_W(n),\quad n\ge 1,
 \end{equation*}
where $c(L)$ and $C(L)$ are positive constants depending only on
$L$.
\end{theorem}

Theorem~\ref{thm:general} is deduced from Theorem~\ref{thm:trig} via
the standard device of defining a $2\pi$-periodic weight function
$W$ from a weight function $w$ on $[-1,1]$ by
$W(\theta):=w\left(\cos\theta\right)|\sin\theta|$.
Theorem~\ref{thm:trig} admits the following heuristic
interpretation: A trigonometric polynomial of degree $n$ `sees' a
doubling weight $W$ at a resolution of $\frac{1}{n}$, as made
precise by Mastroianni and Totik in \cite[Theorem
3.1]{MastroianniTotik1} (see also
Theorem~\ref{thm:MastroianniTotik1} below). Thus one may hope to
obtain a trigonometric quadrature of degree $n$ for $W$ by placing
$n$ nodes at $n$ roughly equi-distanced locations along the unit
circle, in which case each node should, in a sense, replace $W$ in a
neighborhood of size $\frac{1}{n}$ around it. In particular, the
weight of each node should roughly equal the integral of $W$ in that
neighborhood, so that the minimal weight would be of order
$\int_{-\pi}^{\pi} W(\theta)d\theta/R^{\trig}_W(n)$. Such a
trigonometric quadrature may then be converted into a Chebyshev-type
trigonometric quadrature with $N\approx R^{\trig}_W(n)$ nodes by
replacing each of the $n$ nodes by a cluster of nodes of weight
$\int_{-\pi}^{\pi} W(\theta)d\theta/N$, roughly preserving the
weight at each location.

The paper is organized as follows. In Section \ref{sec:prelim} we
list the ingredients we will use in our proofs: a proposition of
Kane and some facts about doubling weight functions.
Theorem~\ref{thm:general} and Theorem~\ref{thm:trig} are proved in
Section \ref{sec:proof_thm}. In Section~\ref{sec:non_doubling} we
remark upon extensions to non-doubling weight functions, providing a
simple upper bound on the minimal number of nodes required in a
Chebyshev-type quadrature and discussing some explicit examples. In
Section~\ref{sec:discussion} we discuss some open problems for
future research.

Throughout the paper we adopt the following policy regarding
constants. The constants $C(L)$ and $c(L)$ will always denote
positive constants whose value depends only on $L$. The
values of these constants will be allowed to change from line to
line, even within the same calculation, with the value of $C(L)$
increasing and the value of $c(L)$ decreasing. We similarly use $C$
and $c$ to denote positive absolute constants whose value may change
from line to line.

\paragraph*{Acknowledgements} We thank Lev Buhovski, Daniel M. Kane, Eli Levin, Mikhail Sodin and Sasha Sodin for fruitful
discussions.

\section{Preliminaries}\label{sec:prelim}

In this section we gather several existing results which we rely
upon in our proofs.

For a non-negative integer $n$, let $\mathcal{T}_n$ denote the set
of real trigonometric polynomials of degree at most $n$, i.e., the
linear space spanned by the $2n+1$ functions
$\{1,\sin\theta,\cos\theta,\sin 2\theta,\cos 2\theta,\ldots,\sin
n\theta,\cos n\theta\}$. Let $\mathcal{T}^{+}_n$ be the subset of
$\mathcal{T}_n$ of non-negative trigonometric polynomials which are
not identically zero.

\subsubsection*{Kane's bound on the size of Chebyshev-type
quadratures}

The following proposition is a restriction of Proposition~20 in
\cite{Kane}, adapted to our setup. For completeness, a proof is
provided in Appendix~\ref{sec:appendix_a}.

\begin{propos}\cite[Proposition 20]{Kane}\label{propos:Kane}
Let $W$ be a $2\pi$-periodic weight function and let $n$ be a
positive integer. Then for every integer $N$ satisfying that
$$N>\frac{\int_{-\pi}^{\pi} W(\theta)d\theta}{2}\sup_{p\in \mathcal{T}^{+}_n}\frac{\int_{-\pi}^{\pi}\left|p'(\theta)\right|d\theta}{\int_{-\pi}^{\pi} p(\theta)W(\theta)d\theta}
$$
there is a Chebyshev-type trigonometric quadrature of degree $n$ for
$W$ with $N$ nodes. Equivalently,
$$\overbar{N}^{\trig}_W(n)\leq\left\lfloor\frac{\int_{-\pi}^{\pi} W(\theta)d\theta}{2}\sup_{p\in \mathcal{T}^{+}_n}\frac{\int_{-\pi}^{\pi}\left|p'(\theta)\right|d\theta}{\int_{-\pi}^{\pi} p(\theta)W(\theta)d\theta}\right\rfloor+1.
$$
\end{propos}

We use this bound together with the classical $L^1$ version of
Bernstein's inequality (see, e.g., \cite[Vol. II, Theorem 3.16,
p.11]{Zygmund}),
\begin{equation}\label{eq:L_1_Bernstein}
\int_{-\pi}^{\pi}\left|p'(\theta)\right|d\theta\leq
n\int_{-\pi}^{\pi}|p(\theta)|d\theta,\quad p\in\mathcal{T}_n.
\end{equation}

\subsubsection*{Properties of doubling weight functions}

Given a $2\pi$-periodic weight function $W$ and a positive integer $n$,
define
\begin{equation}\label{eq:W_averaged}
W_n(x):=n\int_{x-\frac{1}{n}}^{x+\frac{1}{n}}W(\theta)d\theta,\quad
x\in\mathbb{R}.
\end{equation}
Mastroianni and Totik \cite{MastroianniTotik1} showed that the
result of integrating powers of trigonometric polynomials of degree
$n$ against a doubling weight $W$ remains unchanged, up to
constants, if $W$ is replaced by $W_n$. In our work we use only a
one-sided bound of this form and only in the special case when the
power is 1.

\begin{theorem}\cite[Special case of Theorem 3.1]{MastroianniTotik1}\label{thm:MastroianniTotik1}
Let $W$ be a doubling $2\pi$-periodic weight function with doubling
constant $L$. Then
\begin{gather*}\int_{-\pi}^{\pi} \left|p(\theta)\right|W_n(\theta)d\theta\leq C(L)\int_{-\pi}^{\pi}
\left|p(\theta)\right|W(\theta)d\theta,\quad
p\in\mathcal{T}_n.\end{gather*}
\end{theorem}

The following simple lemma allows to transfer results from the
trigonometric to the algebraic setting.
\begin{lemma}\label{lem:w_to_W}
Let $w$ be a weight function on $[-1,1]$. Define a $2\pi$-periodic
weight function $W$ by
\begin{equation}\label{eq:W_from_w_def}
W(\theta):=w(\cos\theta)|\sin\theta|,\quad \theta\in \mathbb{R}.
\end{equation}
If $w$ is doubling with a doubling constant $L$ then $W$ is also
doubling with a doubling constant $C(L)$.
\end{lemma}
\begin{proof}
For every real $u$ and positive $\delta$, let
\begin{align*}
a(u,\delta):&=\min\{\cos(u+t)\mid-\delta\leq t\leq\delta\},\\
b(u,\delta):&=\max\{\cos(u+t)\mid-\delta\leq t\leq\delta\},
\end{align*}
and note that, for every real $t$,
\begin{gather*}
\cos(u+4t)-\cos u=2\cos t(\cos(u+3t)-\cos(u+t)).
\end{gather*}
Therefore $\cos u-a(u,4\delta)\leq 2(b(u,3\delta)-a(u,3\delta))$ and
$b(u,4\delta)-\cos u\leq 2(b(u,3\delta)-a(u,3\delta))$, whence
$b(u,4\delta)-a(u,4\delta)\leq 4(b(u,3\delta)-a(u,3\delta))$.
Applying this bound iteratively three times yields that
\begin{gather*}
b(u,2\delta)-a(u,2\delta)\leq
4^3\left(b\left(u,(3/4)^32\delta\right)-a\left(u,(3/4)^32\delta\right)\right)\leq
4^3\left(b\left(u,\delta\right)-a\left(u,\delta\right)\right),
\end{gather*}
from which it follows that
\begin{gather*}
\left[a(u,2\delta),b(u,2\delta)\right]\subseteq\left[\frac{a(u,\delta)+b(u,\delta)}{2}-2^7\,\frac{b(u,\delta)-a(u,\delta)}{2},\,\frac{a(u,\delta)+b(u,\delta)}{2}+2^7\,\frac{b(u,\delta)-a(u,\delta)}{2}\right].
\end{gather*}
For $\delta<\pi$ we get, using the doubling property of $w$ (and
continuing to interpret $w$ as $0$ outside $[-1,1]$),
\begin{multline*}
\int_{u-2\delta}^{u+2\delta}W(\theta)d\theta=\int_{u-2\delta}^{u+2\delta}w(\cos\theta)|\sin\theta|d\theta\leq
 4\int_{a(u,2\delta)}^{b(u,2\delta)}w(t)dt\leq\\
\leq 4\int_{\frac{a(u,\delta)+b(u,\delta)}{2}-2^7\,\frac{b(u,\delta)-a(u,\delta)}{2}}^{\frac{a(u,\delta)+b(u,\delta)}{2}+2^7\,\frac{b(u,\delta)-a(u,\delta)}{2}}w(t)dt\leq 4L^7\int_{\frac{a(u,\delta)+b(u,\delta)}{2}-\frac{b(u,\delta)-a(u,\delta)}{2}}^{\frac{a(u,\delta)+b(u,\delta)}{2}+\frac{b(u,\delta)-a(u,\delta)}{2}}w(t)dt=\\
=4L^7\int_{a(u,\delta)}^{b(u,\delta)}w(t)dt\leq
4L^7\int_{u-\delta}^{u+\delta}w(\cos\theta)|\sin\theta|d\theta=
4L^7\int_{u-\delta}^{u+\delta}W(\theta)d\theta.
\end{multline*}
Furthermore, it is straightforward that for $\delta\ge \pi$ we have
\begin{equation*}
  \int_{u-2\delta}^{u+2\delta}W(\theta)d\theta\le
  3\int_{u-\delta}^{u+\delta}W(\theta)d\theta,
\end{equation*}
completing the proof of the lemma.
\end{proof}

Lastly, we require the following immediate property of doubling
weight functions.

\begin{lemma}\label{lem:W_to_W_n}
If $W$ is a doubling $2\pi$-periodic weight function with doubling constant $L$, then for every real $x,y$ and $\delta>0$,
$$\int_{y-\delta}^{y+\delta}W(\theta)d\theta\leq L\left(1+\frac{|x-y|}{\delta}\right)^{\log_2L}\int_{x-\delta}^{x+\delta}W(\theta)d\theta.$$
\end{lemma}

\begin{proof}
Let
$r:=\left\lceil\log_2\left(1+\frac{|x-y|}{\delta}\right)\right\rceil$.
Then
$\left[y-\delta,y+\delta\right]\subseteq\left[x-2^r\delta,x+2^r\delta\right]$
and hence, by the doubling property,
\begin{gather*}\int_{y-\delta}^{y+\delta}W(\theta)d\theta\leq\int_{x-2^r\delta}^{x+2^r\delta}W(\theta)d\theta\leq L^r\int_{x-\delta}^{x+\delta}W(\theta)d\theta\le L\left(1+\frac{|x-y|}{\delta}\right)^{\log_2L}\int_{x-\delta}^{x+\delta}W(\theta)d\theta.\qedhere\end{gather*}
\end{proof}

\section{Proofs}\label{sec:proof_thm}

For every integer $m\ge 0$, Let $F_m$ be the following variant of
the Fej\'er kernel,
\begin{equation}\label{eq:Fejer_def}F_m(\theta):=\left(\frac{\sin\frac{(2m+1)\theta}{2}}{(2m+1)\sin\frac{\theta}{2}}\right)^2,\end{equation}
which is a non-negative trigonometric polynomial of degree $2m$. We
frequently use that
\begin{align}
&\text{$F_m(\theta)\leq\min\left\{1,\left(\frac{\pi}{(2m+1)\theta}\right)^2\right\}$
for every $-\pi\leq\theta\leq \pi$}\label{eq:F_m_upper_bound}\\
&\text{$F_m$ is even, and decreasing on
$\left[0,\frac{2\pi}{2m+1}\right]$,}\label{eq:F_m_even_decreasing}\\
&\left(\frac{2}{\pi}\right)^2\leq
F_m\left(\frac{\pi}{2m+1}\right)\le
F_1\left(\frac{\pi}{3}\right)=\left(\frac{2}{3}\right)^2\text{ for
every }m\geq 1,\label{eq:F_m_at_pi_over_2m+1}
\end{align}
where the inequalities in \eqref{eq:F_m_at_pi_over_2m+1} follow by
noting that
$F_m\left(\frac{\pi}{2m+1}\right)=\left((2m+1)\sin\frac{\pi}{2(2m+1)}\right)^{-2}$
decreases with $m$. The following lemma shows that when integrating
powers of $F_m$ against doubling weight functions the main
contribution comes from a small neighborhood of the origin.

\begin{lemma}\label{lemma:Fejer_W_n}
Let $W$ be a doubling $2\pi$-periodic weight function with doubling constant $L$. For every integers $k,\ell,m$ such that $0<|k|\leq m$, $\ell\geq 5\log_2 L$,
$$\int_{\frac{2k-1}{2m+1}\pi}^{\frac{2k+1}{2m+1}\pi} F_m(\theta)^{\ell}W(\theta)d\theta\leq\left(\frac{2}{3|k|}\right)^{\ell}\int_{-\frac{\pi}{2m+1}}^{\frac{\pi}{2m+1}}W(\theta)d\theta.$$
\end{lemma}

\begin{proof}
With no loss of generality assume $k$ is positive. We have
\begin{equation*}
  F_m(\theta)\leq\left(\frac{2}{3k}\right)^2,\quad\;
  \frac{2k-1}{2m+1}\pi\leq\theta\leq\frac{2k+1}{2m+1}\pi.
\end{equation*}
When $\theta\ge \frac{2\pi}{2m+1}$ this follows from
\eqref{eq:F_m_upper_bound} whereas when $\frac{\pi}{2m+1}\le
\theta\le \frac{2\pi}{2m+1}$ this follows from
\eqref{eq:F_m_even_decreasing} and \eqref{eq:F_m_at_pi_over_2m+1}.
Therefore, by
Lemma \ref{lem:W_to_W_n},
\begin{gather*}\int_{\frac{2k-1}{2m+1}\pi}^{\frac{2k+1}{2m+1}\pi}  F_m(\theta)^{\ell}W(\theta)d\theta\leq\left(\frac{2}{3k}\right)^{2\ell}\int_{\frac{2k-1}{2m+1}\pi}^{\frac{2k+1}{2m+1}\pi}W(\theta)d\theta\leq\left(\frac{2}{3k}\right)^{2\ell}L\left(2k+1\right)^{\log_2L}\int_{-\frac{\pi}{2m+1}}^{\frac{\pi}{2m+1}}W(\theta)d\theta,
\end{gather*}
and the lemma follows since for every $k\geq 1$ and $\ell\geq
5\log_2 L$,
\begin{gather*}
\left(\frac{2}{3k}\right)^{\ell}L\left(2k+1\right)^{\log_2L}\leq\left(\frac{2}{3k}\right)^{5\log_2L}L\left(3k\right)^{\log_2L}=\left(\left(\frac{8}{9}\right)^2\frac{1}{k^4}\right)^{\log_2L}\le1.
\qedhere\end{gather*}
\end{proof}

We proceed to prove our main theorems.

\begin{proof}[Proof of Theorem~\ref{thm:trig}]
Denote $I:=\int_{-\pi}^{\pi}W(\theta)d\theta$. Recall the definition of $W_n$ from \eqref{eq:W_averaged}. Note that  $R^{\trig}_W(n)W_n(\theta)\geq I\cdot n$ for every real $\theta$. For every
$p\in \mathcal{T}^{+}_n$, by \eqref{eq:L_1_Bernstein} and Theorem
\ref{thm:MastroianniTotik1},
\begin{gather*}I\int_{-\pi}^{\pi}\left|p'(\theta)\right|d\theta\leq I\cdot n\int_{-\pi}^{\pi}p(\theta)d\theta\leq R^{\trig}_W(n)\int_{-\pi}^{\pi}p(\theta)W_n(\theta)d\theta\leq C(L)R^{\trig}_W(n)\int_{-\pi}^{\pi}p(\theta)W(\theta)d\theta
\end{gather*}
and the upper bound of the theorem follows by Proposition
\ref{propos:Kane}.

To get the lower bound we need to show that if $\theta_1\leq
\theta_2\leq\ldots\leq \theta_N$ are the nodes of a Chebyshev-type
trigonometric quadrature of degree $n$ for $W$, then for every real
$x$, $\int_{x-\frac{1}{n}}^{x+\frac{1}{n}}W(\theta)d\theta\geq
c(L)I/N$.
With no loss of generality we may assume that all nodes
are in $[-\pi,\pi]$ and that $x=0$.

Let $\ell$ be a positive integer and let $m:=\lfloor
n/2(\ell+1)\rfloor$. We first show that if
$\ell\ge 5\log_2(\pi^2L)$ then the interval
$\left[-\frac{\pi}{2m+1},\frac{\pi}{2m+1}\right]$ contains at least
one node of the Chebyshev-type trigonometric quadrature. This is
obvious if $m=0$ so we assume $m\ge 1$ in the following calculation.
Let $p$ be the polynomial
$p(\theta):=F_m(\theta)^{\ell}\left(F_m(\theta)-F_m(\frac{\pi}{2m+1})\right)$.
We note that
\begin{equation}\label{eq:p_negative}
  p(\theta)\le 0,\quad
  \theta\in[-\pi,\pi]\setminus\left[-\frac{\pi}{2m+1}, \frac{\pi}{2m+1}\right]
\end{equation}
as in this range,
\begin{equation*}
  F_m(\theta) =
  \left(\frac{\sin\frac{(2m+1)\theta}{2}}{(2m+1)\sin\frac{\theta}{2}}\right)^2
  \le \left(\frac{1}{(2m+1)\sin\frac{\pi}{2(2m+1)}}\right)^2 =
  F_m\left(\frac{\pi}{2m+1}\right).
\end{equation*}
Therefore, if there is no node of the Chebyshev-type trigonometric
quadrature in the interval
$\left[-\frac{\pi}{2m+1},\frac{\pi}{2m+1}\right]$, then, using that
$\deg p=2m(\ell+1)\leq n$,
$$\int_{-\pi}^{\pi}p(\theta)W(\theta)d\theta=\frac{I}{N}\sum_{j=1}^N p(\theta_j)\leq 0.$$
However, using \eqref{eq:F_m_even_decreasing} and the doubling
property,
\begin{align*}
\int_{-\frac{\pi}{2m+1}}^{\frac{\pi}{2m+1}}p(\theta)W(\theta)d\theta&\geq\int_{-\frac{\pi}{2(2m+1)}}^{\frac{\pi}{2(2m+1)}}p(\theta)W(\theta)d\theta\geq\\
&\geq\left(F_m\left(\frac{\pi}{2(2m+1)}\right)-F_m\left(\frac{\pi}{2m+1}\right)\right)F_m\left(\frac{\pi}{2(2m+1)}\right)^{\ell}\int_{-\frac{\pi}{2(2m+1)}}^{\frac{\pi}{2(2m+1)}}W(\theta)d\theta\geq\\
&\ge
\frac{1}{3}\left(\frac{8}{\pi^2}\right)^{\ell}\frac{1}{L}\int_{-\frac{\pi}{2m+1}}^{\frac{\pi}{2m+1}}W(\theta)d\theta.
\end{align*}
Moreover, by \eqref{eq:p_negative}, \eqref{eq:F_m_upper_bound} and
using Lemma~\ref{lemma:Fejer_W_n}, if $\ell\ge 5\log_2 L$ then for
every integer $k$ such that $0<|k|\leq m$,
\begin{gather*}-\int_{\frac{2k-1}{2m+1}\pi}^{\frac{2k+1}{2m+1}\pi}p(\theta)W(\theta)d\theta\leq F_m\left(\frac{\pi}{2m+1}\right)\int_{\frac{2k-1}{2m+1}\pi}^{\frac{2k+1}{2m+1}\pi} F_m(\theta)^{\ell}W(\theta)d\theta\leq\left(\frac{2}{3|k|}\right)^{\ell}\int_{-\frac{\pi}{2m+1}}^{\frac{\pi}{2m+1}}W(\theta)d\theta,\end{gather*}
whence, for $\ell\geq 5\log_2(\pi^2L)\geq\max\{5\log_2L,\log_2(\pi^2L)/\log_2(12/\pi^2)\}$,
\begin{gather*}
\int_{-\pi}^{\pi}p(\theta)W(\theta)d\theta=\sum_{k=-m}^m\int_{\frac{2k-1}{2m+1}\pi}^{\frac{2k+1}{2m+1}\pi}p(\theta)W(\theta)d\theta\geq\left(\frac{1}{3L}\left(\frac{8}{\pi^2}\right)^{\ell}-2\left(\frac{2}{3}\right)^{\ell}\sum_{k=1}^m\frac{1}{k^2}\right)\int_{-\frac{\pi}{2m+1}}^{\frac{\pi}{2m+1}}W(\theta)d\theta>0.
\end{gather*}
Therefore, if $\ell\ge  5\log_2(\pi^2L)$ then there is at least one node of the Chebyshev-type
trigonometric quadrature in the interval
$\left[-\frac{\pi}{2m+1},\frac{\pi}{2m+1}\right]$. Let
$\theta_{j_0}$ be one such node.

Now, as $\deg F_m^{\ell}=2m\ell\leq n$, it follows, using
\eqref{eq:F_m_even_decreasing}, \eqref{eq:F_m_at_pi_over_2m+1} and
Lemma~\ref{lemma:Fejer_W_n}, that for $\ell\ge 5\log_2(\pi^2L)$,
\begin{multline*}
\frac{I}{N}\left(\frac{2}{\pi}\right)^{2\ell}\leq\frac{I}{N}F_m\left(\frac{\pi}{2m+1}\right)^{\ell}\leq\frac{I}{N}F_m(\theta_{j_0})^{\ell}\leq\frac{I}{N}\sum_{j=1}^N F_m(\theta_j)^{\ell}=\int_{-\pi}^{\pi}F_m(\theta)^{\ell}W(\theta)d\theta=\\
=\sum_{k=-m}^m\int_{\frac{2k-1}{2m+1}\pi}^{\frac{2k+1}{2m+1}\pi}
F_m(\theta)^{\ell}W(\theta)d\theta\leq\left(1+2\left(\frac{2}{3}\right)^{\ell}\sum_{k=1}^m\frac{1}{k^\ell}\right)\int_{-\frac{\pi}{2m+1}}^{\frac{\pi}{2m+1}}W(\theta)d\theta\le
3\int_{-\frac{\pi}{2m+1}}^{\frac{\pi}{2m+1}}W(\theta)d\theta.
\end{multline*}
Taking $\ell = \lceil 5\log_2(\pi^2L)\rceil$, the last inequality and the
doubling property imply that
$\int_{-1/n}^{1/n}W(\theta)d\theta\geq c(L)I/N$,
as required.
\end{proof}

\begin{proof}[Proof of Theorem \ref{thm:general}]
Let $W$ be the $2\pi$-periodic weight function defined by
$W(\theta):=w(\cos\theta)|\sin\theta|$. It is easy to verify that if
$(\theta_j)_{j=1}^N$ are the nodes of a Chebyshev-type trigonometric
quadrature of degree $n$ for $W$, then $(\cos\theta_j)_{j=1}^N$ are
the nodes of a Chebyshev-type quadrature of degree $n$ for $w$, and
conversly, if $(t_j)_{j=1}^N$ are the nodes of a Chebyshev-type
quadrature of degree $n$ for $w$, then $(\arccos
t_j)_{j=1}^N\cup(-\arccos t_j)_{j=1}^N$ are the nodes of a
Chebyshev-type trigonometric quadrature of degree $n$ for $W$ (where
the union is interpreted as a multiset union). Therefore
\begin{equation}\label{eq:N_N^trig}
\frac{1}{2}N^{\trig}_W(n)\leq N_w(n)\leq\overbar{N}_w(n)\leq
\overbar{N}^{\trig}_W(n).
\end{equation}
By Lemma \ref{lem:w_to_W}, $W$ is doubling with doubling constant $C(L)$. Therefore by Theorem \ref{thm:trig},
\begin{equation}\label{eq:trig}
c(L)R^{\trig}_W(n)\leq
N^{\trig}_W(n)\leq\overbar{N}^{\trig}_W(n)\leq C(L)R^{\trig}_W(n).
\end{equation}
For every $u$ let
\begin{align*}
a(u):&=\min\{\cos(u+\delta)\mid-1/n\leq\delta\leq 1/n\},\\
b(u):&=\max\{\cos(u+\delta)\mid-1/n\leq\delta\leq 1/n\},
\end{align*}
then
\begin{equation}\label{eq:integral_cos_comparison}
\int_{a(u)}^{b(u)}w(t)dt\leq\int_{u-\frac{1}{n}}^{u+\frac{1}{n}}w(\cos\theta)|\sin\theta|d\theta\leq
2\int_{a(u)}^{b(u)}w(t)dt.
\end{equation}
Recall the definition of $\Delta_n$ from the statement of the
theorem. For every real $u$ and $-1/n\leq\delta\leq 1/n$,
\begin{align*}
|\cos u-\cos(u+\delta)|&=|(1-\cos\delta)\cos u+\sin\delta\sin u|\leq (1-\cos\delta)+|\sin\delta|\cdot|\sin u|\leq\\
&\leq \frac{\delta^2}{2}+|\delta|\cdot|\sin u|\leq\frac{1}{n}\left(\frac{1}{n}+|\sin u|\right)=\Delta_n(\cos u).
\end{align*}
Therefore,
\begin{equation*}
[a(u),\,b(u)]\subseteq\left[\cos u-\Delta_n(\cos u),\,\cos
u+\Delta_n(\cos u)\right],
\end{equation*}
whence, using
\eqref{eq:integral_cos_comparison},
\begin{equation*}
\int_{u-\frac{1}{n}}^{u+\frac{1}{n}}W(\theta)d\theta=\int_{u-\frac{1}{n}}^{u+\frac{1}{n}}w(\cos\theta)|\sin\theta|d\theta\leq
2\int_{a(u)}^{b(u)}w(t)dt\leq 2\int_{\cos u-\Delta_n(\cos u)}^{\cos
u+\Delta_n(\cos u)}w(t)dt,
\end{equation*}
where here and below we again interpret $w$ as $0$ outside $[-1,1]$.
Therefore, since $\int_{-\pi}^{\pi}
W(\theta)d\theta=\int_{-\pi}^{\pi}
w(\cos\theta)|\sin\theta|d\theta=2\int_{-1}^1 w(t)dt$,
\begin{equation}\label{eq:I_trig<I}
R^{\trig}_W(n)\geq R_w(n).
\end{equation}
For every real $u$,
\begin{align*}
\Delta_n(\cos u)&=\frac{1}{n}\left(\frac{1}{n}+|\sin u|\right)\leq \frac{1}{n}\left(\frac{1}{n}|\cos u|+\left(1 + \frac{1}{n}\right)|\sin u|\right)\le\\
&\le \frac{\pi^2}{2}\left(1-\cos\left(\frac{1}{n}\right)\right)|\cos u|+\pi\sin\left(\frac{1}{n}\right)|\sin u|\le\\
&\le \frac{15}{2}\left[\left(1-\cos\left(\frac{1}{n}\right)\right)|\cos u|+\sin\left(\frac{1}{n}\right)|\sin u|\right]=\\
&=\frac{15}{2}\max\left\{\left|\cos
u-\cos\left(u-\frac{1}{n}\right)\right|,\,\left|\cos
u-\cos\left(u+\frac{1}{n}\right)\right|\right\}\leq
\frac{15}{2}\left( b(u)-a(u)\right).
\end{align*}
Therefore,
\begin{multline*}\left[\cos u-\Delta_n(\cos u),\,\cos u+\Delta_n(\cos u)\right]\subseteq\\
\subseteq\left[\frac{a(u)+b(u)}{2}-16\frac{b(u)-a(u)}{2},\,\frac{a(u)+b(u)}{2}+16\frac{b(u)-a(u)}{2}\right],
\end{multline*}
whence, using the doubling property of $w$ and
\eqref{eq:integral_cos_comparison},
\begin{multline*}
\int_{\cos u-\Delta_n(\cos
u)}^{\cos u+\Delta_n(\cos u)}w(t)dt\leq \int_{\frac{a(u)+b(u)}{2}-16\frac{b(u)-a(u)}{2}}^{\frac{a(u)+b(u)}{2}+16\frac{b(u)-a(u)}{2}}w(t)dt\leq\\
\leq L^4\int_{a(u)}^{b(u)}w(t)dt\leq
L^4\int_{u-\frac{1}{n}}^{u+\frac{1}{n}}w(\cos\theta)|\sin\theta|d\theta=
L^4\int_{u-\frac{1}{n}}^{u+\frac{1}{n}}W(\theta)d\theta .
\end{multline*}
Thus, using again that $\int_{-\pi}^{\pi}
W(\theta)d\theta=2\int_{-1}^1 w(t)dt$,
\begin{equation}\label{eq:I_trig>I}
R^{\trig}_W(n)\leq 2L^4 R_w(n).
\end{equation}
The theorem follows by combining \eqref{eq:N_N^trig},
\eqref{eq:trig}, \eqref{eq:I_trig<I} and \eqref{eq:I_trig>I}.
\end{proof}

\section{Additional results for non-doubling weight functions}\label{sec:non_doubling}

In this section we briefly remark on Chebyshev-type quadratures for
non-doubling weight functions. We shall write $|A|$ to denote the
Lebesgue measure of a Lebesgue measurable set $A$ in the real line.

\subsection{A general upper bound}

Here we provide a simple upper bound on the minimal number of nodes
required in a Chebyshev-type quadrature which is applicable for
\emph{any} weight function. The advantage of this bound is its
generality and the fact that it is sharp in some cases, e.g., for
generalized Jacobi weight functions on $[-1,1]$ (see
\eqref{eq:generalized_Jacobi_weights}) as it has the same order of
magnitude as the bound given in Theorem \ref{thm:general}. Its
disadvantage is that it is not sharp in general, e.g., for weight
functions vanishing on an interval (see also \cite{Forster, Peled})
or for the example given in Section~\ref{sec:non-doubling_example}
below.

\begin{theorem}\label{thm:upper_bound}
Let $n$ be a positive integer and let $W$ be a $2\pi$-periodic
weight function. Then, for every $0<\eta<1$, and every measurable
$D\subset[-\pi,\pi]$ such that $|D|\leq(1-\eta)^2\frac{2\pi}{2n+1}$,
$$\overbar{N}^{\trig}_W(n)\leq  \frac{1}{\eta}\cdot\frac{\int_{-\pi}^{\pi} W(\theta)d\theta}{\essinf_{\theta\in [-\pi,\pi]\setminus D}W(\theta)}n.$$
\end{theorem}

We also provide an analogue of the theorem for weight functions on
$[-1,1]$.

\begin{cor}\label{cor:upper_bound}
Let $n$ be a positive integer and let $w$ be a weight function on
$[-1,1]$. Then, for every $0<\eta<1$, and every measurable
$D\subset[-1,1]$ such that
$\int_D\frac{dt}{\sqrt{1-t^2}}\leq(1-\eta)^2\frac{\pi}{2n+1}$,
$$\overbar{N}_w(n)\leq  \frac{2}{\eta}\cdot\frac{\int_{-1}^1 w(t)dt}{\essinf_{t\in [-1,1]\setminus D}\sqrt{1-t^2}w(t)}n.$$
\end{cor}

Theorem~\ref{thm:upper_bound} follows by combining Proposition
\ref{propos:Kane}, \eqref{eq:L_1_Bernstein} and the following lemma.
The lemma and its proof are inspired by \cite[Lemma 23]{Kane}.

\begin{lemma}\label{lemma:general} Let $n$ be a positive integer and let $W$ be a $2\pi$-periodic weight function. Suppose $D\subset[-\pi,\pi]$ is a
measurable set such that $|D|\leq(1-\eta)^2\frac{2\pi}{2n+1}$ for
some $0<\eta<1$. Then
$$\int_{-\pi}^{\pi} p(\theta)W(\theta)d\theta\geq \eta\essinf_{\theta\in[-\pi,\pi]\setminus D}W(\theta)\cdot\int_{-\pi}^{\pi}p(\theta)d\theta,\quad p\in \mathcal{T}^{+}_n.$$
\end{lemma}

\begin{proof}
It is easy to verify that
$$\int_{-\pi}^{\pi}q(\theta)d\theta=\frac{2\pi}{2n+1}\sum_{j=0}^{2n}q\left(\frac{2j}{2n+1}\pi\right),\quad q\in\mathcal{T}_{2n}.$$
Fix $p\in \mathcal{T}^{+}_n$. As
$\deg p\leq n$ and $p$ is non-negative,
\begin{align*}
\int_{-\pi}^{\pi} p(\theta)^2d\theta&=\frac{2\pi}{2n+1}\sum_{j=0}^{2n}p\left(\frac{2j}{2n+1}\pi\right)^2\leq\frac{2\pi}{2n+1}\left(\sum_{j=0}^{2n}p\left(\frac{2j}{2n+1}\pi\right)\right)^2=\\
&=\frac{2n+1}{2\pi}\left(\frac{2\pi}{2n+1}\sum_{j=0}^{2n}p\left(\frac{2j}{2n+1}\pi\right)\right)^2=\frac{2n+1}{2\pi}\left(\int_{-\pi}^{\pi}
p(\theta)d\theta\right)^2.
\end{align*}
Therefore, by the Cauchy-Schwartz inequality,
\begin{align*}\int_D p(\theta)d\theta&\leq\sqrt{\int_D d\theta}\cdot\sqrt{\int_D p(\theta)^2d\theta}\leq\sqrt{|D|}\cdot\sqrt{\int_{-\pi}^{\pi} p(\theta)^2d\theta}\leq\\
&\leq \sqrt{(1-\eta)^2\frac{2\pi}{2n+1}}\cdot
\sqrt{\frac{2n+1}{2\pi}}\int_{-\pi}^{\pi}
p(\theta)d\theta=(1-\eta)\int_{-\pi}^{\pi} p(\theta)d\theta.
\end{align*}
Hence
\begin{multline*}
\int_{-\pi}^{\pi}
p(\theta)W(\theta)d\theta\geq\int_{[-\pi,\pi]\setminus D}
p(\theta)W(\theta)d\theta\geq \essinf_{\theta\in [-\pi,\pi]\setminus D}W(\theta)\cdot\int_{[-\pi,\pi]\setminus D} p(\theta)d\theta\geq\\
\geq \eta\essinf_{\theta\in[-\pi,\pi]\setminus
D}W(\theta)\cdot\int_{-\pi}^{\pi} p(\theta)d\theta.
\qedhere\end{multline*}
\end{proof}
\begin{proof}[Proof of Corollary~\ref{cor:upper_bound}]
Let $W$ be the $2\pi$-periodic weight function defined by
$W(\theta):=w(\cos\theta)|\sin\theta|$ and set $\tilde{D}$ to be the
subset of $[-\pi,\pi]$ satisfying $-\tilde{D} = \tilde{D}$ and
$\cos(\tilde{D}) = D$. The corollary follows from
Theorem~\ref{thm:upper_bound} by observing that
$\overbar{N}_w(n)\leq \overbar{N}^{\trig}_W(n)$ by
\eqref{eq:N_N^trig}, that $|\tilde{D}| =
2\int_D\frac{dt}{\sqrt{1-t^2}}$, that $\essinf_{\theta\in
[-\pi,\pi]\setminus \tilde{D}}W(\theta) = \essinf_{t\in
[-1,1]\setminus D}\sqrt{1-t^2}w(t)$ and that $\int_{-\pi}^{\pi}
W(\theta)d\theta = 2\int_{-1}^1 w(t)dt$.
\end{proof}

\subsection{Exponentially vanishing weights}\label{sec:non-doubling_example}

In this section we study a family of non-doubling weight functions
which vanish as a stretched exponential at a point. For $\alpha>0$,
let $W_{\alpha}$ be the $2\pi$-periodic weight function defined for
$-\pi\leq\theta<\pi$ by
\begin{equation}\label{eq:W_2}W_{\alpha}(\theta):=\begin{cases}e^{-|\theta|^{-\alpha}}& \theta\neq 0\\ 0 &\theta=0\end{cases}.\end{equation}
Theorem~\ref{thm:trig} does not apply to this weight function as it
is not doubling. Theorem \ref{thm:upper_bound} applies, and yields
that $\overbar{N}^{\trig}_{W_{\alpha}}(n)\leq \exp(Cn^\alpha)$. As
it turns out, however, $\overbar{N}^{\trig}_{W_{\alpha}}(n)$ is
considerably smaller.

\begin{theorem}\label{thm:stretched_exponential}
Let $W_{\alpha}$ be the $2\pi$-periodic weight function defined for
$-\pi\leq\theta<\pi$ by \eqref{eq:W_2}. Then
\begin{equation*}
\exp\left(c(\alpha)n^{\frac{\alpha}{\alpha+1}}\right)\leq
N^{\trig}_{W_{\alpha}}(n) \leq
\overbar{N}^{\trig}_{W_{\alpha}}(n)\leq
\exp\left(C(\alpha)n^{\frac{\alpha}{\alpha+1}}\right),\quad n\ge 1,
\end{equation*}
where $C(\alpha),c(\alpha)$ are positive constants depending only on
$\alpha$.
\end{theorem}

We mention that the technique used for the proof of the theorem uses
similar ideas to the proof of Theorem~\ref{thm:trig} and may be
applicable for certain other weight functions. For instance, for the
$2\pi$-periodic weight functions $\widetilde{W}_{\alpha}$,
$\alpha>0$, defined for $0<|\theta|\le\pi$ by
$\widetilde{W}_{\alpha}(\theta):=\exp(-\exp(|\theta|^{-\alpha}))$ we
may show that
\begin{equation*}
\exp\left(c(\alpha)\frac{n}{(\log n)^{1/\alpha}}\right)\leq
N^{\trig}_{\widetilde{W}_{\alpha}}(n) \leq
\overbar{N}^{\trig}_{\widetilde{W}_{\alpha}}(n)\leq
\exp\left(C(\alpha)\frac{n}{(\log n)^{1/\alpha}}\right),\quad n\ge
1,
\end{equation*}
where $C(\alpha),c(\alpha)$ are positive constants depending only on
$\alpha$. It is also worth mentioning in this regard that it is
known \cite[Theorem 1.4]{Peled} that for weight functions $w$ on
$[-1,1]$ which are bounded we have $\overbar{N}_w(n)\le \exp(C(w)n)$
for a constant $C(w)$ depending only on $w$.

As a second remark we point out that the fact that the bounds
provided by Theorem~\ref{thm:stretched_exponential} for
$N^{\trig}_{W_{\alpha}}(n)$ and
$\overbar{N}^{\trig}_{W_{\alpha}}(n)$ are not as close to each other
as in Theorem~\ref{thm:trig} may be essential.
Indeed, as mentioned before, there is an example \cite{Forster} of a
weight function $w$ vanishing on an interval for which $N_w(n)$ is
linear in $n$ while $\overbar{N}_w(n)$ is exponential in $n$.

The proof of Theorem~\ref{thm:stretched_exponential} is obtained by
combining Proposition \ref{propos:Kane} with the following
Remez-type inequality for trigonometric polynomials.

\begin{theorem}\label{thm:Remez}\cite[Theorem 5.1.2]{Borwein-Erdelyi}
Let $p\in\mathcal{T}_n$ and denote
$M:=\max_{\theta\in\mathbb{R}}|p(\theta)|$. Then
$$|\{-\pi\leq\theta<\pi: |p(\theta)|\geq Me^{-4ns}\}|\geq s,\quad 0<s\leq\pi/2.$$
\end{theorem}

\begin{proof}[Proof of Theorem~\ref{thm:stretched_exponential}]
We first prove the upper bound. We may assume
$3n^{-\frac{1}{\alpha+1}}\leq \pi/2$. Fix $p\in\mathcal{T}^+_n$ and
denote $M:=\max_{\theta\in\mathbb{R}}|p(\theta)|$. Applying
Theorem~\ref{thm:Remez} with $s=3n^{-\frac{1}{\alpha+1}}$ yields
that
$$\left|\left\{-\pi\leq\theta<\pi: |p(\theta)|\geq M\exp\left(-12n^{\frac{\alpha}{\alpha+1}}\right)\right\}\right|\geq 3n^{-\frac{1}{\alpha+1}}.$$
Let
$$A:=\left\{-\pi\leq\theta<\pi: |p(\theta)|\geq M\exp\left(-12n^{\frac{\alpha}{\alpha+1}}\right)\right\}\setminus\left(-n^{-\frac{1}{\alpha+1}},n^{-\frac{1}{\alpha+1}}\right).$$
It follows that $|A|\geq n^{-\frac{1}{\alpha+1}}$ and
$W_{\alpha}(\theta)\geq
\exp\left(-n^{\frac{\alpha}{\alpha+1}}\right)$ for every $\theta\in
A$, whence
\begin{align*}
\int_{-\pi}^{\pi} p(\theta)W_{\alpha}(\theta)d\theta&\geq\int_A
p(\theta)W_{\alpha}(\theta)d\theta\geq
n^{-\frac{1}{\alpha+1}}M\exp\left(-12n^{\frac{\alpha}{\alpha+1}}\right)\exp\left(-n^{\frac{\alpha}{\alpha+1}}\right)\geq\\
&\ge
n^{-\frac{1}{\alpha+1}}\exp\left(-13n^{\frac{\alpha}{\alpha+1}}\right)\frac{1}{2\pi}\int_{-\pi}^{\pi}
p(\theta)d\theta.
\end{align*}
By Proposition \ref{propos:Kane} and \eqref{eq:L_1_Bernstein} we get that
$$\overbar{N}^{\trig}_{W_{\alpha}}(n)\leq C\,n^{1+\frac{1}{\alpha+1}}\exp\left(13n^{\frac{\alpha}{\alpha+1}}\right)$$
and the upper bound follows.

To get the lower bound we need to show that if $\theta_1\leq
\theta_2\leq\ldots\leq \theta_N$ are the nodes of a Chebyshev-type
trigonometric quadrature of degree $n$ for $W_{\alpha}$, then $\log
N\geq c(\alpha)\,n^{\frac{\alpha}{\alpha+1}}$. Fix such a
quadrature. With no loss of generality we may assume that all nodes
are in $[-\pi,\pi]$. Let
$m:=\lfloor\frac{1}{12}6^{\frac{\alpha}{\alpha+1}}n^{\frac{1}{\alpha+1}}\rfloor$,
$\ell:=\lfloor
6^{\frac{1}{\alpha+1}}n^{\frac{\alpha}{\alpha+1}}-1\rfloor$,
$r:=\lfloor
\min\{6^{\frac{1}{\alpha+1}},\alpha/6^{\alpha+1}\}n^{\frac{\alpha}{\alpha+1}}\rfloor$.
We may assune that $n$, and hence also $m,\ell$ and $r$ are large
enough, as functions of $\alpha$, for the following calculation.

As in the proof of Theorem \ref{thm:trig} we first show that  the interval
$\left[-\frac{\pi}{2m+1},\frac{\pi}{2m+1}\right]$ contains at least
one node of the Chebyshev-type trigonometric quadrature.
Let $p$ be the polynomial
$p(\theta):=F_m(\theta)^{\ell}\left(F_m(\theta)-F_m(\frac{\pi}{2m+1})\right)$.  Note that $\deg p=2m(\ell+1)\leq n$. If there is no node of the Chebyshev-type trigonometric quadrature in the interval
$\left[-\frac{\pi}{2m+1},\frac{\pi}{2m+1}\right]$ then, as in the proof of Theorem \ref{thm:trig}, $\int_{-\pi}^{\pi}p(\theta)W_{\alpha}(\theta)d\theta\leq 0$.
However, using \eqref{eq:F_m_even_decreasing},
\begin{align*}\int_0^{\frac{\pi}{2m+1}}&p(\theta)W_{\alpha}(\theta)d\theta\geq\int_{\frac{\pi}{3(2m+1)}}^{\frac{2\pi}{3(2m+1)}}p(\theta)W_{\alpha}(\theta)d\theta\geq\\
&\geq\frac{\pi}{3(2m+1)}\left(F_m\left(\frac{2\pi}{3(2m+1)}\right)-F_m\left(\frac{\pi}{2m+1}\right)\right)F_m\left(\frac{2\pi}{3(2m+1)}\right)^{\ell}W_{\alpha}\left(\frac{\pi}{3(2m+1)}\right)\geq\\
&\geq\frac{c}{2m+1}\left(\frac{3\sqrt{3}}{2\pi}\right)^{2\ell}\exp\left(-\left(\frac{3(2m+1)}{\pi}\right)^{\alpha}\right)\geq
\left(\frac{3}{4}\right)^{2\ell}.
\end{align*}
In the last inequality we have used that
\begin{gather*}
\lim_{n\to\infty}
\frac{1}{2\ell}\left(\frac{3(2m+1)}{\pi}\right)^{\alpha}=\frac{1}{12}\left(\frac{3}{\pi}\right)^{\alpha}<\frac{1}{12}
\end{gather*}
together with the fact that $\frac{3\sqrt{3}}{2\pi}e^{-1/12} >
\frac{3}{4}$.
We also have, using
\eqref{eq:F_m_upper_bound},\eqref{eq:F_m_even_decreasing} and
\eqref{eq:F_m_at_pi_over_2m+1}, that
\begin{gather*}-\int_{\frac{\pi}{2m+1}}^{\pi}p(\theta)W_{\alpha}(\theta)d\theta\leq \int_{\frac{\pi}{2m+1}}^{\pi}F_m(\theta)^\ell d\theta\leq
\frac{\pi}{2m+1}F_m\left(\frac{\pi}{2m+1}\right)^\ell+
\int_{\frac{2\pi}{2m+1}}^{\pi}\left(\frac{\pi}{(2m+1)\theta}\right)^{2\ell}d\theta\leq\\
\leq \frac{\pi}{2m+1}\left(\frac{2}{3}\right)^{2\ell}+
\frac{(2m-1)\pi}{2m+1}\left(\frac{1}{2}\right)^{2\ell}\leq
\pi\left(\frac{2}{3}\right)^{2\ell},\end{gather*} whence
\begin{gather*}
\int_{-\pi}^{\pi}p(\theta)W_{\alpha}(\theta)d\theta=2\int_{0}^{\pi}p(\theta)W_{\alpha}(\theta)d\theta>0.
\end{gather*}
Therefore there is at least one node of the Chebyshev-type
trigonometric quadrature in the interval
$\left[-\frac{\pi}{2m+1},\frac{\pi}{2m+1}\right]$. Let
$\theta_{j_0}$ be one such node. Let
$I:=\int_{-\pi}^{\pi}W_{\alpha}(\theta)d\theta$. Now, since
$\deg\left((F_m)^r\right)=2rm\leq n$, it follows, using
\eqref{eq:F_m_even_decreasing}, \eqref{eq:F_m_at_pi_over_2m+1} and
\eqref{eq:F_m_upper_bound}, that
\begin{multline*}
\frac{I}{N}\left(\frac{2}{\pi}\right)^{2r}\leq\frac{I}{N}\,F_m\left(\frac{\pi}{2m+1}\right)^r\leq\frac{I}{N}\,F_m(\theta_{j_0})^r\leq\frac{I}{N}\sum_{j=1}^N F_m(\theta_j)^r=\int_{-\pi}^{\pi}F_m(\theta)^r W_{\alpha}(\theta)d\theta\leq\\
\leq
\int_{-\pi}^{\pi}\left(\frac{\pi}{(2m+1)\theta}\right)^{2r}e^{-\frac{1}{|\theta|^{\alpha}}}d\theta\leq
2\pi\left(\frac{\pi}{2m+1}\left(\frac{2r}{e\alpha}\right)^{1/\alpha}\right)^{2r}<
2\pi\left(\frac{\pi}{6}\right)^{2r},
\end{multline*}
where in the penultimate inequality we used that
$\left(\frac{2r}{e\alpha}\right)^{2r/\alpha}$ is the maximum of
$\theta^{-2r}\exp\left(-|\theta|^{-\alpha}\right)$ and in the last
inequality we used that
$$\frac{\pi}{2m+1}\left(\frac{2r}{e\alpha}\right)^{1/\alpha}\leq \frac{\pi}{2m+1}\left(\frac{2\cdot\frac{\alpha}{6^{\alpha+1}}\cdot
n^{\frac{\alpha}{\alpha+1}}}{e\alpha}\right)^{1/\alpha}\xrightarrow[n\to\infty]{}
\frac{\pi}{6}\left(\frac{6^{\frac{\alpha}{\alpha+1}}}{3e}\right)^{1/\alpha}<\frac{\pi}{6}.$$
Therefore $N> \frac{I}{2\pi}\left(\frac{12}{\pi^2}\right)^{2r}$ and
the lower bound follows.
\end{proof}

\section{Discussion and open questions}\label{sec:discussion}

In our work we find the order of magnitude of the minimal number of
nodes required in Chebyshev-type quadratures for doubling weight
functions, and also briefly discuss more general weight functions.
It is natural to wonder whether any of the qualitative phenomena
observed for doubling weight functions are in fact true in greater
generality. In this section we briefly discuss some questions of
this type. The questions are formulated in the trigonometric case
but are natural also on the interval.

\subsubsection*{Multiplication by a function}

Let $h$ be a $2\pi$-periodic measurable function satisfying that
there exist $M,m>0$ such that
\begin{equation*}
m<h(x)<M
\end{equation*}
almost everywhere. Let $W$ be a $2\pi$-periodic weight function. Is
it true that the minimal number of nodes required in Chebyshev-type
quadratures for $W$ and for $hW$ is of the same order of magnitude?
Precisely, that there exist $C(M,m),c(M,m)>0$, depending only on $M$
and $m$ (in fact, only on $M/m$ by homogeneity), such that
\begin{align}
  c(M,m) N^{\trig}_{W}(n) \le\, &N^{\trig}_{hW}(n) \le C(M,m)
  N^{\trig}_{W}(n),\quad n\ge 1,\label{eq:multiplication_question}\\
  c(M,m) \overbar{N}^{\trig}_{W}(n) \le\, &\overbar{N}^{\trig}_{hW}(n) \le C(M,m)
  \overbar{N}^{\trig}_{W}(n),\quad n\ge
  1.\label{eq:multiplication_question_overbar}
\end{align}
Our results show that this phenomenon holds for doubling weight
functions.

\subsubsection*{Sharpness of Kane's bound}

Our upper bounds on $\overbar{N}^{\trig}_{W}(n)$ are based on Kane's
general upper bound given in Proposition~\ref{propos:Kane}. In fact,
we do not know any example of a $2\pi$-periodic weight function
which is non-zero almost everywhere for which this bound is not
sharp up to constants. Does it give the correct order of magnitude
of $\overbar{N}^{\trig}_{W}(n)$ (as $n$ tends to infinity) for all
such weight functions? A positive answer will also verify the
relation \eqref{eq:multiplication_question_overbar} discussed in the
previous question.

It is necessary to make some assumption on the support of the weight
function. For instance, as pointed out to us by Kane, if $W$ is the
$2\pi$-periodic weight function defined for $-\pi\leq\theta<\pi$ by
\begin{equation*}W(\theta):=\begin{cases}1& |\theta|\le \frac{\pi}{6}\\ 0 &|\theta|>\frac{\pi}{6}\end{cases}.\end{equation*}
For even $n$, the bound on $\overbar{N}^{\trig}_{W}(n)$ given by
Proposition~\ref{propos:Kane} is at least $2^n$, by considering the
polynomial $p(\theta) = (\sin\theta)^n\in\mathcal{T}^{+}_n$.
However, it is in fact true that $cn^2\le N^{\trig}_W(n)\le
\overbar{N}^{\trig}_{W}(n)\le C n^2$. To see this, observe that
$W(\theta) = w(\cos\theta)|\sin\theta|$ for the weight function $w$
on $[-1,1]$ defined by
\begin{equation*}
  w(x):=\begin{cases}\frac{1}{\sqrt{1-x^2}}& x\ge \frac{\sqrt{3}}{2}\\ 0 &x<\frac{\sqrt{3}}{2}\end{cases}.
\end{equation*}
A Chebyshev-type trigonometric quadrature for $W$ may be obtained
from a Chebyshev-type quadrature for $w$ (as in the proof of
Theorem~\ref{thm:general}) by mapping the nodes $(t_j)_{j=1}^N$ to
the nodes $(\arccos t_j)_{j=1}^N\cup(-\arccos t_j)_{j=1}^N$ (where
the union is interpreted as a multiset union). Now, $cn^2\le
N_{w}(n)\leq\overbar{N}_{w}(n)\le Cn^2$ as follows from
Theorem~\ref{thm:general} applied to $w$ composed with an affine
mapping.

\subsubsection*{Locality}
Theorem~\ref{thm:trig} shows that the order of magnitude of
$N^{\trig}_{W}(n)$ and $\overbar{N}^{\trig}_{W}(n)$ is given by
$R^{\trig}_W(n)$ for doubling $2\pi$-periodic weight functions $W$.
The order of magnitude of $R^{\trig}_W(n)$, as a function of $n$, is
determined by the expression
$\inf_{x\in\mathbb{R}}\int_{x-\frac{1}{n}}^{x+\frac{1}{n}}W(\theta)d\theta$.
This gives a sense to the idea that the orders of magnitude of
$N^{\trig}_{W}(n)$ and $\overbar{N}^{\trig}_{W}(n)$ depend only on
local features of $W$. Is this property shared also by non-doubling
weight functions?

To make this question precise, suppose $W_1,W_2,\ldots,W_m$ and $W$
are $2\pi$-periodic weight functions such that for every real $x$,
$W$ coincides with some $W_i$ in a neighbourhood of $x$. Is it true
that there exists a constant $M$, depending only on
$W_1,W_2,\ldots,W_m$ and $W$, for which
\begin{align*}
&N^{\trig}_{W}(n)\leq M\max_{1\leq i\leq m}N^{\trig}_{W_i}(n),\\
&\overbar{N}^{\trig}_{W}(n)\leq M\max_{1\leq i\leq
m}\overbar{N}^{\trig}_{W_i}(n)
\end{align*}
for every $n\ge 1$? One cannot expect similar inequalities in the
opposite direction to hold, even in the doubling case, as no control
is provided on $W_i$ away from the neighbourhoods where it coincides
with $W$ and Theorem~\ref{thm:trig} shows that changing $W_i$ in
these regions can increase $N^{\trig}_{W_i}(n)$ and
$\overbar{N}^{\trig}_{W_i}(n)$ significantly.

\subsubsection*{Numerical algorithms}
Our results provide bounds on the minimal number of nodes in
Chebyshev-type quadratures for doubling weight functions. From a
numerical analysis point of view it is of interest to complement
these with efficient algorithms for finding such quadrature
formulas. It appears that this problem has not received much
attention in the literature and we are not aware of efficient
algorithms for it. Research has been focused on identifying in the
set of all Chebyshev-type quadratures with a given number of nodes
and a given degree of accuracy, such formulas which are optimal in a
specified sense. The interested reader is pointed to the recent book
of Brass and Petras \cite[Chapter 9]{Brass-Petras}, to the older
survey of Gautschi \cite{Gautschi_Survey}, to the works of Anderson
and Gautschi \cite{Anderson_and_Gautschi} and of F\"orster and
Ostermeyer \cite{Forster_and_Ostermeyer} and to the references
there, for discussion of such notions and their properties. The work
of Anderson and Gautschi suggests also an algorithm for finding a
quadrature formula with certain optimality properties.

\subsubsection*{Extensions}
Both the notions of quadrature formula and the notion of doubling
are, in fact, defined for the measure given by $w(t)dt$ (or
$W(\theta)d\theta$ in the trigonometric case). Thus these notions
may be extended naturally to general finite, positive measures. With
these extended notions, Theorem~\ref{thm:general},
Theorem~\ref{thm:trig} and their proofs continue to apply, mutatis
mutandis, for the class of doubling measures.

Some authors extend the notion of quadrature formula further,
allowing the nodes to be outside of $[-1,1]$ for weight functions on
$[-1,1]$, or allowing the nodes to take complex values for
$2\pi$-periodic weight functions. We do not consider these extended
notions here.

\appendix

\section{Proof of Proposition \ref{propos:Kane}}\label{sec:appendix_a}

In presenting the proof of Proposition \ref{propos:Kane},
essentially following the proof in \cite{Kane}, we use the following
three lemmas.

\begin{lemma}\label{lem:analysis}
Let $W$ be a $2\pi$-periodic weight function and let $n$ and $N$ be
positive integers, such that
\begin{equation}\label{eq:N_cond_in_lemma}
2N>\int_{-\pi}^{\pi}W(\theta)d\theta\sup_{p\in
\mathcal{T}^{+}_n}\frac{\int_{-\pi}^{\pi}\left|p'(\theta)\right|d\theta}{\int_{-\pi}^{\pi}
p(\theta)W(\theta)d\theta}.
\end{equation}
Then there is a finite set $S\subset[-\pi,\pi)$ such that every
non-zero $q\in\mathcal{T}_n$ for which $\int_{-\pi}^{\pi}
q(\theta)W(\theta)d\theta=0$ satisfies
\begin{equation}\label{eq:N_q_max_inequality}
  2N\max_{x\in S}
  q(x)>\int_{-\pi}^{\pi}\left|q'(\theta)\right|d\theta.
\end{equation}
\end{lemma}

\begin{proof}
We consider $\mathcal{T}_n$ as a subset of the space of continuous
$2\pi$-periodic functions endowed with the supremum norm. Let
$V:=\{r\in\mathcal{T}_n\mid \int_{-\pi}^{\pi}
r(\theta)W(\theta)d\theta=0\}$ and observe that
\begin{equation}\label{eq:non-zero_in_V}
  \text{a non-zero $r\in V$ satisfies $\max_x r(x)>0$}
\end{equation}
as non-zero trigonometric polynomials cannot be zero on a set of
positive measure. Let $K:=\{r\in V\mid\max_{x}r(x)=1\}$. We first
prove that $K$ is a compact set, for which it suffices to show that
\begin{equation}\label{eq:compactness_criterion}
  \sup_{r\in K} \max_x |r(x)|<\infty.
\end{equation}
Let $B:=\{q\in V\mid \max_x |q(x)|=1\}$. The set $B$ is compact as a
closed and bounded set in the finite-dimensional space
$\mathcal{T}_n$. Thus the continuous functional $f\mapsto\max_x
f(x)$ attains a minimum on $B$, which must be positive due to
\eqref{eq:non-zero_in_V}. This implies
\eqref{eq:compactness_criterion}.

Take $\eta>1$ such that $2N$ is still bigger than $\eta$ times the
right-hand side of \eqref{eq:N_cond_in_lemma}. For every $-\pi\leq
x<\pi$ let $U_x:=\{r\in\mathcal{T}_n\mid r(x)>1/\eta\}$. Noting that
$\{U_x\}_{-\pi\leq x<\pi}$ is an open cover of the compact set $K$
we conclude that there is a finite set $S\subset [-\pi,\pi)$ such
that $K\subseteq\cup_{x\in S}U_x$. Consequently, using
\eqref{eq:non-zero_in_V},
\begin{equation}\label{eq:max_on_set}
\max_x q(x) \le \eta\max_{x\in S}q(x),\quad q\in V.
\end{equation}

Finally, let $q\in V$ be a non-zero trigonometric polynomial in $V$.
Define $p_0:=\eta(\max_{x\in S}q(x))-q$ and note that
$p_0\in\mathcal{T}^{+}_n$ by \eqref{eq:max_on_set} and $\max_{x\in
S} q(x)>0$ by \eqref{eq:non-zero_in_V} and \eqref{eq:max_on_set}.
Thus,
\begin{equation*}2N>\eta\int_{-\pi}^{\pi}W(\theta)d\theta\sup_{p\in \mathcal{T}^{+}_n}\frac{\int_{-\pi}^{\pi}\left|p'(\theta)\right|d\theta}{\int_{-\pi}^{\pi} p(\theta)W(\theta)d\theta}\geq\eta\int_{-\pi}^{\pi}W(\theta)d\theta\frac{\int_{-\pi}^{\pi}\left|p_0'(\theta)\right|d\theta}{\int_{-\pi}^{\pi} p_0(\theta)W(\theta)d\theta}=\frac{\int_{-\pi}^{\pi}\left|q'(\theta)\right|d\theta}{\max_{x\in s}q(x)}.
\qedhere\end{equation*}
\end{proof}

\begin{lemma}\label{lem:geometry}
Let $v_1,v_2,\ldots,v_R$ be points in ${\mathbb R}^n$, no $n+1$ of
which are on the same affine hyperplane, and let $Q\subset{\mathbb
R}^n$ be their convex hull. Let $-\pi\leq x_1<x_2<\ldots<x_R< \pi$
and let $N$ be a positive integer. Then there are continuous
functions $f_1,f_2,\ldots,f_N$ from $Q$ to $[-\pi,\pi)$ such that if
$v$ belongs to a facet of $Q$ whose vertices are
$v_{i_1},v_{i_2},\ldots,v_{i_n}$, $1\leq i_1<i_2<\ldots<i_n\leq R$,
then $\{f_1(v),f_2(v),\ldots,f_N(v)\}\subset[x_{i_1},x_{i_n}]$ and
in each of the intervals
$(x_{i_1},x_{i_2}),\ldots,(x_{i_{n-1}},x_{i_n})$ there is at most
one member of the set $\{f_1(v),f_2(v),\ldots,f_N(v)\}$.
\end{lemma}

\begin{proof} Take some triangulation of $Q$. Let $v\in Q$. If $v$ is in the simplex of the chosen triangulation whose vertices are $v_{i_0},v_{i_1},\ldots,v_{i_n}$ then $v$ can be uniquely presented as a convex combination $v=\sum_{j=0}^n \alpha_j v_{i_j}$. For every $-\pi\leq x< \pi$  let
$$\rho_v(x):=\frac{\pi+x}{2\pi}+N\sum_{\substack{0\leq j\leq n\\x_{i_j}\leq x}}\alpha_j$$
and observe that this definition is independent of the choice of the
simplex of the triangulation which contains $v$ (if more than one
exists). Define $f_i(v):=\min\{-\pi\leq x< \pi\mid \rho_v(x)\geq
i\}$. It is easy to verify that each $f_i$ is continuous. Let $v$ be
a point on a facet of $Q$ whose vertices are
$v_{i_1},v_{i_2},\ldots,v_{i_n}$, $1\leq i_1<i_2<\ldots<i_n\leq R$.
Then $\rho_v(x)<1$ for every $-\pi\leq x<x_{i_1}$ and $\rho_v(x)>N$
for every $x_{i_n}<x< \pi$, whence
$\{f_1(v),f_2(v),\ldots,f_N(v)\}\subset[x_{i_1},x_{i_n}]$. Also
$\rho_v(y)-\rho_v(x)<1$ for every $x_{i_{k-1}}<x<y<x_{i_k},\,1<k\leq
n$, whence in each of the intervals
$(x_{i_1},x_{i_2}),\ldots,(x_{i_{n-1}},x_{i_n})$ there is at most
one member of the set $\{f_1(v),f_2(v),\ldots,f_N(v)\}$.
\end{proof}

\begin{lemma}\cite[Proposition 7]{Kane}\label{lem:topology}
Let $Q\subset {\mathbb R}^n$ be a convex polytope with $0$ in its
interior. Let $F:Q\to{\mathbb R}^n$ be a continuous function such
that if $\{v\in{\mathbb R}^n\mid u\cdot v=1\}$ is an affine
hyperplane containing a facet $T$ of $Q$ then $u\cdot F(v)>0$ for
every $v\in T$. Then there is a $v\in Q$ such that $F(v)=0$.
\end{lemma}

\begin{proof} By way of contradiction, assume that $F(v)\neq 0$ for every $v\in Q$.
Let $B^n:=\{x\in{\mathbb R}^n\mid |x|\leq 1\}$ be the closed unit
ball of ${\mathbb R}^n$. For every $u\in{\mathbb R}^n$ let $\lVert
u\rVert_Q:=\min\{\lambda>0\mid \frac{1}{\lambda}u\in Q\}$ and define
$h:B^n\to B^n$ by
$$h(x):=\begin{cases}-\frac{F\left(\frac{|x|}{\lVert x\rVert_Q}x\right)}{\left| F\left(\frac{|x|}{\lVert x\rVert_Q}x\right)\right|}\quad &x\neq 0\\
-\frac{F\left(0\right)}{\left| F\left(0\right)\right|}\quad &x=0
\end{cases}.
$$
The function $h$ is continuous, so by Brouwer's fixed point theorem,
there exists an $x_0\in B^n$ such that $h(x_0)=x_0$. Let
$y_0:=\frac{|x_0|}{\lVert x_0\rVert_Q}x_0$. Since $|h(x)|=1$ for
every $x\in B^n$ we get that $|x_0|=|h(x_0)|=1$, and therefore
$y_0=\frac{|x_0|}{\lVert x_0\rVert_Q}x_0=\frac{x_0}{\lVert
x_0\rVert_Q}$  is on the boundary of $Q$. Therefore, $y_0$ belongs
to at least one facet $T$ of $Q$. Suppose $T$ is contained in the
affine hyperplane $\{x\in{\mathbb R}^n\mid u\cdot x=1\}$. Then
$u\cdot F(y_0)>0$ and we get a contradiction since
\begin{multline*}
1=u\cdot y_0=u\cdot \frac{x_0}{\lVert x_0\rVert_Q}=\frac{1}{\lVert x_0\rVert_Q}\,u\cdot x_0=\frac{1}{\lVert x_0\rVert_Q}\,u\cdot h(x_0)=\\
=\frac{1}{\lVert x_0\rVert_Q}\,u\cdot \left(-\frac{F(y_0)}{\left|
F(y_0)\right|}\right)=-\frac{1}{\lVert x_0\rVert_Q\left|
F(y_0)\right|}\,u\cdot F(y_0)<0. \qedhere\end{multline*}
\end{proof}
\begin{proof}[Proof of Proposition \ref{propos:Kane}]
Suppose $N$ is a positive integer such that
$$2N>\int_{-\pi}^{\pi} W(\theta)d\theta\sup_{p\in \mathcal{T}^{+}_n}\frac{\int_{-\pi}^{\pi}\left|p'(\theta)\right|d\theta}{\int_{-\pi}^{\pi} p(\theta)W(\theta)d\theta}
$$
and let $S$ be the finite set provided by Lemma~\ref{lem:analysis}.
Let $\varphi_0\equiv 1,\varphi_1,\ldots,\varphi_{2n}$ be a basis of
$\mathcal{T}_n$. For every $-\pi\leq x< \pi$ let $E(x)$ be the
vector
$\left(\left(\int_{-\pi}^{\pi}
W(\theta)d\theta\right)\varphi_i(x)-\int_{-\pi}^{\pi}\varphi_i(\theta)W(\theta)d\theta\right)_{i=1}^{2n}$
in ${\mathbb R}^{2n}$. It is straightforward to check that no $2n+1$
elements of $\{E(x)\}_{-\pi\le x<\pi}$, are on the same affine
hyperplane. Define $Q$ to be the convex hull of
$\{E\left(x\right)\}_{x\in S}$. We claim that $0$ is in the interior
of $Q$. Indeed, otherwise there exists a non-zero
$u\in\mathbb{R}^{2n}$ for which $u\cdot E(x) \le 0$ for all $x\in
S$. However, the non-zero trigonometric polynomial in
$\mathcal{T}_n$ defined by $p_u(x):=u\cdot E(x)$ satisfies
$\int_{-\pi}^\pi p_u(\theta)W(\theta)d\theta = 0$ and thus, by
\eqref{eq:N_q_max_inequality}, we have that $\max_{x\in S}
p_u(x)>0$, a contradiction.

Suppose $S=\{x_1,x_2,\ldots,x_R\}$ where $-\pi\leq
x_1<x_2<\ldots<x_R<\pi$. Let $f_1,f_2,\ldots,f_N$ be the continuous
functions from $Q$ to $[-\pi, \pi)$ guaranteed by
Lemma~\ref{lem:geometry}.
Let $v$ be a point on a facet of $Q$ whose vertices are
$E\left(x_{i_1}\right),E\left(x_{i_2}\right),\ldots,E\left(x_{i_{2n}}\right)$,
where $1\leq i_1<i_2<\ldots<i_{2n}\leq R$. Let $\{w\in{\mathbb
R}^n\mid u\cdot w=1\}$ be the affine hyperplane containing the facet
(using that $0$ is in the interior of $Q$). Let $q$ be the non-zero
trigonometric polynomial in $\mathcal{T}_n$ defined by $q(x):=u\cdot
E(x)$, so that $\int_{-\pi}^{\pi} q(\theta)W(\theta)d\theta=0$. Note
that $q(x)\leq 1$ for every $x\in S$ as $0$ is in the interior of
$Q$, whence $2N>\int_{-\pi}^{\pi}\left|q'\theta)\right|d\theta$ by
Lemma~\ref{lem:analysis}. Let $J:=\{1\leq j\leq N\mid
f_j(v)\notin\{x_{i_1},x_{i_2},\ldots,x_{i_{2n}}\}\}$. Note that
Lemma~\ref{lem:geometry} implies that for every $j\in J$ there is a
$2\leq k_j\leq 2n$ such that $f_j(v)\in(x_{i_{k_j-1}},x_{i_{k_j}})$
and these $k_j$ are distinct. Then, since $q\left(x_{i_k}\right)=1$
for every $1\leq k\leq 2n$,
\begin{multline*}
N>\frac{1}{2}\int_{-\pi}^{\pi}\left|q'(\theta)\right|d\theta\geq
\frac{1}{2}\sum_{j\in J}\left(\left|\int_{x_{i_{k_j-1}}}^{f_j(v)}q'(t)dt\right|+\left|\int_{f_j(v)}^{x_{i_{k_j}}}q'(t)dt\right|\right)=\\
=\sum_{j\in J}\frac{\left|q\left(f_j(v)\right)-q(x_{i_{k_j-1}})\right|+\left|q(x_{i_{k_j}})-q\left(f_j(v)\right)\right|}{2}=\\
=\sum_{j\in
J}\left|1-q\left(f_j(v)\right)\right|=\sum_{j=1}^N\left|1-q\left(f_j(v)\right)\right|\geq
N-\sum_{j=1}^N q\left(f_j(v)\right),
\end{multline*}
whence $u\cdot\sum_{j=1}^N E\left(f_j(v)\right) = \sum_{j=1}^N
q\left(f_j(v)\right)>0$.
Thus, if we define a function $F:Q\to\mathbb{R}^{2n}$ by
$F(v):=\sum_{j=1}^N E\left(f_j(v)\right)$ then $F$ satisfies the
assumptions of Lemma~\ref{lem:topology}.
Consequently, there is some $v\in Q$ such that $\sum_{j=1}^N
E\left(f_j(v)\right)=F(v)=0$, i.e., $\{f_j(v)\}_{j=1}^N$ are the
nodes of a Chebyshev-type trigonometric quadrature of degree $n$ for
$W$.
\end{proof}

\end{document}